\documentclass{lms}
\usepackage{amsmath,amsfonts,amssymb,array,cite,prettyref}
\usepackage[colorlinks=true,backref=true,linkcolor=blue,citecolor=blue]{hyperref}
\usepackage[all]{xy}
\usepackage[refpage]{nomencl}

\DeclareMathOperator{\frob}{Fr}
\DeclareMathOperator{\tr}{Tr}

\newcommand{\lam}{\lambda}
\newcommand{\Lam}{\Lambda}
\newcommand{\Om}{\Omega}
\newcommand{\om}{\omega}

\newcommand{\Sig}{\Sigma}

\newcommand{\vphi}{\varphi}

\newcommand{\Cal}[1]{\mathcal{#1}}
\newcommand{\bb}[1]{\mathbb{#1}}

\newcommand{\mfrak}[1]{\mathfrak{#1}}
\newcommand{\wbar}[1]{\overline{#1}}

\newcommand{\rcomp}{\backslash}

\newcommand{\norm}[1]{\left\|\, #1 \,\right\|}
\newcommand{\ang}[1]{\langle #1 \rangle}
\newcommand{\paren}[1]{\left( #1 \right)}

\newcommand{\abs}[1]{\left| #1 \right|}

\newcommand{\End}{\textup{End}}

\newcommand{\Gal}{\textup{Gal}}
\newcommand{\Irr}{\textup{Irr}}
\newcommand{\sep}{\textup{sep}}
\newcommand{\spec}{\textup{Spec}\,}

\newcommand{\SL}[2]{\textup{SL}_{#1}(#2)}
\newcommand{\GL}[2]{\textup{GL}_{#1}(#2)}

\newcommand{\Sp}[2]{\textup{Sp}_{#1}(#2)}

\newcommand{\GSp}[2]{\textup{GSp}_{#1}(#2)}

\newcommand{\On}[3]{\textup{O}_{#1}^{#2}(#3)}

\newtheorem{thm}{Theorem}[section]
\newtheorem{lma}[thm]{Lemma}
\newtheorem{cor}[thm]{Corollary}
\newtheorem{prty}[thm]{Property}

\newnumbered{defn}{Definition}

\newnumbered{rmk}{Remark}

\newrefformat{eq}{\textup{(\ref{#1})}}
\newrefformat{cond}{\textup{(\ref{#1})}}

\makenomenclature

\title[Principally polarized abelian surfaces]
      {Principally polarized abelian surfaces with surjective Galois representations on $l$-torsion}

\author{Erik Wallace}

\classno{11G10, 11F80 (primary), 11N36 (secondary)}

\begin{document}
\maketitle

\begin{abstract}
 Given a rational variety $V$ defined over $K$, we consider a principally polarized abelian variety $A$
 of dimension $g$ defined over the function field $K(V)$. For each prime $l$ we then consider the Galois
 representation on the $l$-torsion of $A_t$, where $t$ is a $K$-rational point of $V$. The largest possible image is
 $\GSp{2g}{l}$, and in the cases $g=1$ and 2 we are able to attain this image for all
 $l$ and almost all $t$. In the case $g=1$ this recovers a theorem originally proven by William Duke \cite{Duke}. 
\end{abstract}

\section{Introduction}
A well known result of Duke \cite{Duke} states that if $\Cal{C}(X)$ denotes the set of equivalence classes elliptic curves over $\bb{Q}$ of the form $y^2=x^3+rx+s$
such that $\max\{|r|^3,|s|^2\}\leq X^6$, and $\Cal{E}(X)$ denotes the subset for which the representation of $\Gal(\wbar{\bb{Q}}/\bb{Q})$ on the
$l$-torsion of a representative elliptic curve is not surjective for some $l$, then
\[
 \lim_{X\to \infty} \frac{|\Cal{E}(X)|}{|\Cal{C}(X)|}=0.
\]
Zywina \cite{Zyw1} has proven a generalization to arbitrary number fields, however, so far there have been no extensions
to higher dimensional abelian varieties. In this paper we prove a generalization to abelian surfaces.

Zywina's result is accomplished with a version of the large sieve proven by Serre \cite{Ser2}, which is applied for each $l$,
and each conjugacy class in $\GL{2}{\bb{F}_l}$. Another interesting observation of Zywina is that the Siegel-Walfisz theorem
is not actually necessary.  However, Zywina does not consider elliptic curves over the function field of a curve, and so he is
forced to use a result of Jones \cite{Jon2} that requires the Eichler-Selberg trace formula. We avoid this by using a version
of the Chebotarev density theorem for function fields first proven by Lang \cite{Lang1}, which we sharpen slightly.

In section \ref{sec1} we begin discussing a version of the large sieve compatible with heights, as well as an application of a Brun
type sieve to obtain a lower bound for abelian varieties with good reduction.  This result on good reduction was required in an earlier
version of this paper, but it is still included as an interesting result in its own right. A large portion of our results are not
restricted by dimension in any way, and can possibly be applied to cases other than abelian varieties.  So in section \ref{sec2} we
discuss the results that can be proven in a very general way, the main result being Theorem \ref{general-thm}. This theorem is applied
in section \ref{sec4} to abelian varieties of dimension 1 and 2.  The result can be stated as follows:

\begin{thm}
Let $K$ be a number field, let $V$ be smooth geometrically irreducible affine variety over $K$ of dimension $r=\binom{g+1}{2}$,
birationally equivalent to $\bb{P}_K^r$ via the rational map $\vphi:V\to\bb{P}_K^r$.  Let $K(V)$ be the function field of $V$,
and let $A$ be a principally polarized abelian variety over $K(V)$ of dimension $g=1$ or 2.  If we define
\begin{align*}
 B_K(x)&=B_K(x)=\{t\in V(K):t\in U\text{ for some }\ang{U,\vphi_U},\text{ and }H_\vphi(t)\leq x \}\text{ and,}\\
 E_K(x)&=\{t\in B_K(x): \rho_{l,t}(G_K)\subsetneq \GSp{2g}{l}\text{ for some $l$}\},
\end{align*}
then
\[
 \frac{|E_K(x)|}{|B_K(x)|}\ll_{K,r,\vphi} \frac{(\log x)^{2(g^2+g+2)\gamma+1}}{x^\frac{1}{2}}.
\]
\end{thm}

The result in dimension 2 is new, but in dimension 1 it is little more than a hybrid of previous results.
A list of notation can be found in appendix \ref{notation}.  As a bonus we also show in section \ref{sec1} that if
\[
 C_K(x;Q)=\{t\in B_K(x):A_t\text{ has good reduction for all }\mfrak{p}\in \Cal{L}^*\},
\]
where $\Cal{L}$ is a subset of the prime ideals $\mfrak{p}$ of $\Cal{O}_K$, having $N(\mfrak{p})$ a prime $\leq Q$, then
we have a lower bound of the form
\[
 |C_K(x)|\gg \frac{x^{r+1}}{(\log Q)^\kappa}
\]
for some $\kappa\geq 0$.  Unfortunately, the relationship between $x$ and $Q$ is not sharp.

\subsection{Acknowledgments}
This paper comprises a portion of my dissertation. I would like to thank my advisor Michael Larsen for his guidance, and Aner Shalev for directing us to the work of
Kleidman and Liebeck.  I would also like to express my sincere thanks to the reviewer, who did an extremely thorough reading of this paper, and made some very
helpful suggestions.  For instance, the previous version used this lower bound for abelian varieties having good reduction at the primes in $\Cal{L}^*$, but because
the relationship between $x$ and $Q$ is not sharp, the estimate given previously was weaker.  The reviewer pointed out a way that this might be avoided, and it is
this new approach we have implemented in the article.  

\section{Sieve theory}\label{sec1}

The sieve methods we will use need to be compatible with heights on number fields.
This is accomplished using the following construction which can be compared with the
method of Schanuel \cite{Schanuel}.  Let $K$ \nomenclature{$K$}{a number field}
be a number field of degree $d$ and ring of integers $\Cal{O}_K$,
\nomenclature{$\Cal{O}_K$}{the ring of integers of $K$} and let
$S_\infty$ \nomenclature{$S_\infty$}{the set of infinite places of $K$}
denote the set of infinite places of $K$.  If $v\in S_\infty$, then
\[
 \norm{\cdot}_v=|\cdot|_v^{[K_v:\bb{R}]/d}
\]
defines a norm on $K_v$. \nomenclature{$K_v$}{the completion of $K$ at $v$}
If $u=(u_0:u_1:\cdots :u_r)\in \bb{P}_K^r(K)$, and $\mfrak{a}_u$ denotes the fractional ideal generated by the $u_i$, then we have
\[
 H(u)=N(\mfrak{a}_u)^{-\frac{1}{d}}\prod_{v\in S_\infty} \sup_i\norm{u_i}_v
\]
where $H$ is the absolute height. By scaling, it is possible to obtain coordinates $u_i$ in $\Cal{O}_K$. Making this choice
minimally gives us a representative in a fundamental domain of $K^{r+1}-0^{r+1}$ under the action of units.
Let $\Lam$ be the image of $\Cal{O}_K$ under the diagonal embedding $K\to\prod_{v\in S_\infty} K_v$,
\nomenclature{$\Lam$}{the image of $\Cal{O}_K$ under the diagonal embedding $K\to\prod_{v\in S_\infty} K_v$}
and let 
\[
 B_K(x)=\{u\in\bb{P}_K^r(K): H(u)\leq x\}.
\]
With the coordinates of each $u$ chosen as above, we first lift to $ K^{r+1}-0^{r+1}$ and then consider the image in $\prod_{v\in S_\infty} K_v^{r+1}$ under
the diagonal embedding, as illustrated by the following diagram
\[
 \bb{P}_K^{r}(K)\longleftarrow K^{r+1}-0^{r+1}\longrightarrow \prod_{v\in S_\infty} K_v^{r+1}.
\]
It is contained in $\Lam^{r+1}$, and by the results of Schanuel its size grows at a rate proportional to
that of a ball of radius $x$ in $\Lam^{r+1}$.  This will allow us to apply Serre's version of the large sieve (see \cite{Ser2} or \cite{Zyw1}).

In \cite{Kow2} Kowalski has given a language for sieves that we would like to generalize slightly,
and also use to discuss the constructions in our particular application. Then we will prove a version
of the large sieve that is compatible with heights on affine space, and an application of the
lower bound sieve that we will use with it.

A \emph{sieve setting} is a triple $(Y,\Cal{A},(\pi_\alpha))$ \nomenclature{$(Y,\Cal{A},(\pi_\alpha))$}{abstract notation for a sieve setting}
consisting of a set $Y$, an indexing set $\Cal{A}$, and for each $\alpha\in\Cal{A}$ a map
$\pi_\alpha:Y\to Y_\alpha$, where $Y_\alpha$ is a finite set.  For both the large sieve and the lower bound sieve,
we will take $Y=\Lam^{r+1}$.  The indexing set $\Cal{A}$ will be different in each case though: for the large sieve
we will take it to be $\Sig_K$,\nomenclature{$\Sig_K$}{the set of non-zero prime ideals in $\Cal{O}_K$}
the set of non-zero prime ideals in $\Cal{O}_K$, whereas for the lower bound sieve we will use the ordinary primes,
which we will view as ideals in $\Cal{O}_K$. If $\mfrak{a}$ is an arbitrary ideal in $\Cal{O}_K$, then it can be identified with
a sub-lattice of $\Lam$, and the quotient is isomorphic to $\Cal{O}_K/\mfrak{a}$; thus we obtain a natural map
$\pi_\mfrak{a}:\Lam^{r+1}\to(\Cal{O}_K/\mfrak{a})^{r+1}$. For the lower bound sieve, when we have $\mfrak{a}=(p)$, we denote this map by $\pi_p$.
If $\mfrak{a}=\mfrak{p}_1\mfrak{p}_2\cdots \mfrak{p}_k$ is square-free, then by the Chinese remainder theorem, we get an isomorphism
\[
 \vphi_\mfrak{a}:(\Cal{O}_K/\mfrak{p}_1)^{r+1}\oplus (\Cal{O}_K/\mfrak{p}_2)^{r+1}\oplus\cdots \oplus (\Cal{O}_K/\mfrak{p}_k)^{r+1} \to (\Cal{O}_K/\mfrak{a})^{r+1},
\]
such that the map $\pi_\mfrak{a}$ is compatible with the maps $\pi_{\mfrak{p}_i}$ for all $i$. 

A \emph{siftable set} is a triple $(X,\mu,F)$, \nomenclature{$(X,\mu,F)$}{abstract notation for a siftable set}
consisting of a finite measure space $(X,\mu)$, and a map $F:X\to Y$, such that the composite map $\pi_\alpha\circ F$ is measurable.
This introduces a lot of flexibility, because $X$ does not have to be a subset of $\Lam^{r+1}$.  The discussion at the beginning of
this section shows how to construct a map $F:X\to \Lam^{r+1}$, in the case where $X$ is a finite subset of $\bb{P}_K(K)$. If $\mu$
is the counting measure on $X$, then it is clear that the composite maps $\pi_\mfrak{p}\circ F$ are measurable, so this gives us a siftable set.
Now if we have a finite morphism $\vphi:U_K\to \bb{P}_K^r$, we get a height $H_\vphi$ on $U_K$ \nomenclature{$H_\vphi$}{a height on a variety
with respect to the map $\vphi$} by taking the composition $H\circ \vphi$ (we mostly follow \cite{Silv} for the general theory of heights,
but \cite{Ser2} is also useful), and then we can look at a subset $X$ of
\[
 B_K(x)=\{u\in U_K(K):H_\vphi(u)\leq x \}.
\]
Since the morphism $\vphi$ is finite, it follows that the set $X$ is finite, and so by composition with the map $\bb{P}_K^r(K)\to \Lam^{r+1}$ constructed above, we get a map
$F:X\to\Lam^{r+1}$. Also, note that since $\bb{P}_K^r$ has finite type over $K$ and $\vphi$ is a finite morphism, then $U_K$ also has finite type over $K$. This means that $U_K$
can be extended to a scheme over $\spec\Cal{O}_K[\frac{1}{\alpha}]$ for some $\alpha\in\Cal{O}_K$.  This enables us to consider reduction mod $\mfrak{p}$ for all but a finite
number of $\mfrak{p}\in\Sig_K$.

But we can go even further than this. If $\vphi:V_K\to \bb{P}_K^r$ \nomenclature{$V_K$}{a variety over $K$}
is a finite rational map, then by definition we have equivalence classes of pairs $\ang{U,\vphi_U}$ such that $\vphi_U:U\to \bb{P}_K^r$ is a finite morphism,
and hence we can apply the argument above to this morphism. Given two pairs $\ang{U_1,\vphi_1}$ and $\ang{U_2,\vphi_2}$, the compatibility condition gives us
\[
 H_{\vphi_1}|_{U_1\cap U_2}=H_{\vphi_2}|_{U_1\cap U_2},
\]
where $H_{\vphi_1},H_{\vphi_2}$ are the heights corresponding to $\vphi_1,\vphi_2$.  This allows us to safely speak of
a height $H_\vphi$ corresponding to the rational function $\vphi$, although it may not be defined on all of $V_K$. If we make the assumption that $V_K$ has finite type over $K$,
then it is noetherian, so the domain of the rational map has a finite cover by open sets $U_1,\ldots U_N$.  Additionally, the finite type assumption allows us to extend to a scheme
over $\spec\Cal{O}_K[\frac{1}{\alpha}]$ as noted above. In general, suppose $(X,\mu)$ is a finite measure space, but we don't have a map $F:X\to Y$.
If we have a finite cover $X$ by the sets $X_1,\ldots X_N$, such that $(X_i,\mu,F_i)$ is a siftable set, then we can use sub-additivity to extend the large sieve so that
it applies to $(X,\mu)$, and thus we would like to call this a siftable set also. In particular, we can regard any subset $X$ of
\[
 B_K(x)=\{t\in V_K(K):t\in U\text{ for some }\ang{U,\vphi_U},\text{ and }H_\vphi(t)\leq x \}
\]
as a siftable set.

A \emph{prime sieve support} $\Cal{L}^*$, \nomenclature{$\Cal{L}^*$}{the prime sieve support}
is a finite subset of the indexing set $\Cal{A}$, and a \emph{sieve support} $\Cal{L}$\nomenclature{$\Cal{L}$}{the sieve support} is a subset of the power set of $\Cal{L}^*$.
Let $\Sig_K(Q)$\nomenclature{$\Sig_K(Q)$}{$\{\mfrak{p}\in \Sig_K: N(\mfrak{p})< Q\}$} and
$\Sig_K^1(Q;1,l)$\nomenclature{$\Sig_K^1(Q;1,l)$}{$\{\mfrak{p} \in \Sig_K(Q):N(\mfrak{p})=q\text{ is prime, and } q\equiv 1 \mod l\}$} be defined as follows:
\begin{align*}
\Sig_K(Q)&=\{\mfrak{p}\in \Sig_K: N(\mfrak{p})< Q\}\\
\Sig_K^1(Q;1,l)&=\{\mfrak{p} \in \Sig_K(Q):N(\mfrak{p})=q\text{ is prime, and } q\equiv 1 \mod l\}.
\end{align*}
In our general statement of the large sieve, we can take $\Cal{L}^*$ to be any subset of $\Sig_K(Q)$, but in our application of we will use an appropriate subset of
$\Sig_K^1(Q;1,l)$. Note that any subset of $\Cal{L}^*$ can be viewed as a square-free product of these prime ideals. In particular we will define $\Cal{L}$, to be the set of all
square-free ideals $\mfrak{a}$, such that $N(\mfrak{a})< Q$ and for all $\mfrak{p}|\mfrak{a}$ we have
$\mfrak{p}\in \Cal{L}^*$.  We will need compatibility, between the primes used in the lower bound sieve and the prime ideals used in the large sieve.
If $\mfrak{p}\in \Sig_K^1(Q;1,l)$ and $N(\mfrak{p})=p$, then $p< Q$ hence these are the primes that will be used in the lower bound sieve. Now for each
$\alpha\in \Cal{L}^*$  we choose a \emph{sieving set} $\Om_\alpha\subset Y_\alpha$, which may be completely arbitrary.  If
$(X,\mu,F)$ is a siftable set under Kowalski's original definition, then the \emph{sifted set} is \nomenclature{$S(X,(\Om_\alpha);\Cal{L}^*)$}{abstract notation for the sifted set}
\[
 S(X,(\Om_\alpha);\Cal{L}^*)=\{x\in X| \pi_\alpha(F(x))\notin\Om_\alpha\;\forall\alpha\in\Cal{L}^*\},
\]
or more generally if $(X,\mu)$ is a siftable set under our extended definition, and $(X_i,\mu,F_i)$ is a finite cover by siftable sets under the original definition, then
the sifted set is
\[
 S(X,(\Om_\alpha);\Cal{L}^*)=\{x\in X|\exists i\text{ such that }x\in X_i\text{ and } \pi_\alpha(F_i(x))\notin\Om_\alpha\;\forall\alpha\in\Cal{L}^*\}.
\]
Finally, we take $\nu_\mfrak{p}$ to be the uniform probability measure on $Y_\mfrak{p}$.

\subsection{The large sieve}
We will prove a version of the large sieve for projective varieties, with the help of an older version for torsion free $\Cal{O}_K$-modules proven by Serre
\cite{Ser2}, which we restate below in notation compatible with that described above (see also Zywina's paper \cite{Zyw1}).
\begin{thm}
 Let $\Lam$ be a torsion free $\Cal{O}_K$-module with rank $r+1$ over $\Cal{O}_K$.  Let $\norm{\cdot}$ be a norm of $\Lam_\bb{R}=\bb{R}\otimes \Lam$.
Let $x\geq 1$ and $Q>0$ be real numbers, and for each $\mfrak{p}\in\Sig_K$, let $\om_\mfrak{p}\in[0,1]$.  Suppose that $E\subset \Lam$ satisfies the conditions:
\begin{enumerate}
 \item $E$ is contained in a ball of radius proportional to $x$.
 \item For every $\mfrak{p}$ with $N(\mfrak{p})\leq Q$, we have the inequality
\[
 |\pi_\mfrak{p}(E)|\leq (1-\om_\mfrak{p})|\Lam/\mfrak{p}\Lam|.
\]
\end{enumerate}
Then we have
\[
 |E|\ll_{K,\Lam,\norm{\cdot}} \frac{\max\{x^{(r+1)d},Q^{2(r+1)}\}}{L(Q)}
\]
where the implied constant depends only on $K$, $\Lam$, and $\norm{\cdot}$, and where
\[
 L(Q)=\sum_{\mfrak{a}}\prod_{\mfrak{p}|\mfrak{a}}\frac{\om_\mfrak{p}}{1-\om_\mfrak{p}},
\]
the sum being over all square-free ideals $\mfrak{a}$ with norm $\leq Q$.
\end{thm}
In particular the constant does not depend on $x,Q$ or the numbers $\om_\mfrak{p}$, so long as $E$ satisfies the conditions in the statement of the theorem.
This means that it depends only on the sieve setting, which in our application remains the same for all $l$. Also it is important to realize,
that the set $E$ is actually the sifted set. The statement of the large sieve for projective varieties is as follows.

\begin{thm}[Large Sieve]\label{thm-sieve}
Suppose we have a finite rational map $\vphi:V_K\to \bb{P}_K^r$ and that $V_K$ has finite type over $K$. Let $(X,\mu)$ be a siftable set for the sieve setting
$(\Lam^{r+1},\Sig_K, (\pi_\mfrak{p}))$, such that $X$ is contained in 
\[
B_K(x)=\{t\in V_K(K):t\in U\text{ for some }\ang{U,\vphi_U},\text{ and }H_\vphi(t)\leq x \}.
\]
Let $\Cal{L}^*$ be an arbitrary subset of $\Sig_K(Q)$ and let $(\Om_\mfrak{p})$ be an arbitrary family of sieving sets. Then
\begin{equation}\label{eq:3}
 |S(X,(\Om_\mfrak{p});\Cal{L}^*)|\ll_{K,r,\vphi} \frac{\max\{x^{(r+1)[K:\bb{Q}]},Q^{2(r+1)}\}}{L(Q)}
\end{equation}
where the implied constant depends only on $K,r,\vphi$, and
\begin{equation}\label{eq:4}
L(Q)=\sum_{\mfrak{a}\in\Cal{L}}\prod_{\mfrak{p}|\mfrak{a}}\frac{\nu_\mfrak{p}(\Om_p)}{\nu_\mfrak{p}(Y_\mfrak{p}-\Om_\mfrak{p})}
\end{equation}
where $\Cal{L}$ is the set of all square-free ideals $\mfrak{a}$, such that $N(\mfrak{a})\leq Q$ and for all $\mfrak{p}|\mfrak{a}$ we have $\mfrak{p}\in \Cal{L}^*$.
\end{thm}

\begin{rmk*}
 The theorem is stated for relative heights. It is a simple matter to obtain a version of this theorem for absolute heights by observing that we have $H=H_K^{1/[K:\bb{Q}]}$,
 where $H_K$ is the height relative to $K$ and $H$ is the absolute height.  This means that the  theorem remains correct for absolute heights if we replace the numerator
 of the fraction on the right hand side of \prettyref{eq:3} by $\max\{x^{r+1},Q^{2(r+1)}\}$.  It should also be noted that the implied constant depends only on data from the
 sieve setting, $K,r$, and on data from the siftable set, $\vphi$. In particular it does not depend in any way on $\Cal{L}^*$.
\end{rmk*}

\begin{proof}
By the noetherian condition we may assume that the domain of $\vphi$ has a finite cover $U_1,U_2,\ldots U_N$. Let $\vphi_i$ be the restriction of $\vphi$ to $U_i$, and define
\[
 X_i=X\cap U_i(K)\quad\text{and}\quad F_i:X_i\to \Lam^{r+1},
\]
where $\Lam$ is the Minkowski lattice corresponding to $K$, and the map $F_i$ is constructed as above. As noted above, Serre's theorem is stated for the sifted set,
hence for convenience we define
\[
 E_i=S(X_i,(\Om_\mfrak{p});\Cal{L}^*).
\]

Now $\Lam^{r+1}$ is a torsion free $\Cal{O}_K$-module of rank $r+1$, but we must still verify that conditions \emph{1} and \emph{2} hold.
The discussion at the beginning of the section shows that $E_i$ is contained in a ball of radius proportional to $x$, which shows that condition \emph{1} holds.
As for condition \emph{2}, suppose that $\mfrak{p}\in\Sig_K$ has norm $N(\mfrak{p})\leq Q$. If $\mfrak{p}\in\Cal{L}^*$, then
\[
 \pi_\mfrak{p}(E_i)\subset Y_\mfrak{p}-\Om_\mfrak{p}
\]
where $Y_\mfrak{p}=(\Lam/\mfrak{p}\Lam)^{r+1}$.  Even if $\mfrak{p}\notin\Cal{L}^*$ we can still consider the images of $E_i$, which will be trivially contained
in $Y_\mfrak{p}$.
Therefore, if we define
\[
 \omega_\mfrak{p}=\begin{cases}
                   \nu_\mfrak{p}(\Omega_\mfrak{p})& \text{if }\mfrak{p}\in \Cal{L}^*\\
		   0		& \text{otherwise}
                  \end{cases}
\]
then
\[
\nu_\mfrak{p}(\pi_\mfrak{p}(E_i))\leq  1-\om_\mfrak{p}
\]
in all cases.  Since $\nu_\mfrak{p}$ is the uniform measure on $Y_\mfrak{p}$, this gives us condition \emph{2}.  By Serre's theorem,
we then have
\[
 |E_i|\ll_{\Lam^{r+1},\norm{\cdot}}\frac{\max\{x^{(r+1)[K:\bb{Q}]},Q^{2(r+1)}\}}{L(Q)}
\]
where the implied constant depends only on $\Lam^{r+1}$ and $\norm{\cdot}$, and
\[
 L(Q)=\sum_{\mfrak{a}}\prod_{\mfrak{p}|\mfrak{a}}\frac{\om_\mfrak{p}}{1-\om_\mfrak{p}}
\]
where the sum is over all square-free ideals such that $N(\mfrak{a})\leq Q$.  By the definition of $\om_\mfrak{p}$ it follows that this definition of $L(Q)$ is equivalent
with the one in the statement of the theorem.  In our case, the chosen norm is determined completely by $\Lam^{r+1}$, and $\Lam^{r+1}$ depends only on $K$ and on $r$, so
the implied constant really only depends on $K$ and $r$, and we indicate this by changing the subscript in the inequality.
We get from $E_i$ back to $S(X_i,(\Om_\mfrak{p});\Cal{L}^*)$, via the maps
\[
\xymatrix{
 V_K(K)\ar[r]^\vphi& \bb{P}_K^r(K)\ar[r]& \Lam^{r+1}.
 }
\]
The second map is injective, so it follows that $|E_i|=|S(\vphi(X_i),(\Om_\mfrak{p});\Cal{L}^*)|$, but the map $\vphi$ does not have to be injective. Since $\vphi$ is
finite, at the very least we have
\[
 |S(X_i,(\Om_\mfrak{p});\Cal{L}^*)|\ll_\vphi |S(\vphi(X_i),(\Om_\mfrak{p});\Cal{L}^*)|,
\]
and finally to get back to $S(X,(\Om_\mfrak{p});\Cal{L}^*)$, we use finite sub-additivity.
\end{proof}

\subsection{The lower bound sieve}
Iwaniec and Kowalski have established the following version of Brun's sieve.
\begin{thm}\label{Iwan-Kow}
Let $\kappa>0,D>1$ and $\alpha(u)\geq 0$ for all $u\in X$. There exist upper and lower-bound sieve coefficients $(\lam_d^\pm)$ depending only on
$\kappa$ and $D$, supported on square-free integers $<D$, bounded by one in absolute value, with the following properties: for all
$s\geq 9 \kappa + 1$ and $Q^{9\kappa +1}<D$, we have
\begin{align*}
 \int_{S(X,(\Om_p);Q)}\alpha(u)\,d\mu(u)&<(1+e^{9\kappa+1-s}C^{10})\prod_{p< Q}(1-\nu_p(\Om_p))H+ R^+(X;Q^s),\\
 \int_{S(X,(\Om_p);Q)}\alpha(u)\,d\mu(u)&>(1-e^{9\kappa+1-s}C^{10})\prod_{p< Q}(1-\nu_p(\Om_p))H- R^-(X;Q^s)
\end{align*}
provided that
\begin{equation}\label{eq:14}
 \prod_{w\leq p<Q} \frac{1}{1-\nu_p(\Om_p)}\leq C \paren{\frac{\log Q}{\log w}}^\kappa,
\end{equation}
for all $w$ and $Q$ satisfying $2\leq w<Q<D$.
\end{thm}
In this theorem $H$ and $R^\pm(X;Q^s)$ are defined as follows
\begin{align*}
 H&=\int_X\alpha(x)\,d\mu(x)&
 S_d(X;\alpha)&=\nu_d(\Om_d)H+r_d(X;\alpha)\\
 P(Q)&=\mathop{\prod_{p<Q}}_{l\in\Cal{L}^*} l&
 R^\pm(X;Q^s)&=\mathop{\sum_{d<Q^s}}_{d|P(Q)}|\lam_d^\pm r_d(X;\alpha)|.
\end{align*}
Thus $H$ can be regarded as the main term and $r_d$ or $R^\pm$ are regarded as remainder terms. We apply this to
the sieve setting $(\Lam^{r+1},\Sig_\bb{Q},(\pi_p))$, and siftable set $(B_K(x),\mu,F)$, where
\[
 B_K(x)=\{u\in \bb{P}_K^r(K):H(u)\leq x \},
\]
where $\mu$ is the counting measure. 

We also define
\[
 B_\Lam(x)=\left \{a\in\Lam^{r+1}:\prod_{v\in S_\infty} \sup_i \norm{a_{iv}}_v\leq x\right\},
\]
where $i$ ranges from 0 to $r$, and we denote the uniform probability measure on $B_\Lam(x)$ by $P$.  By construction, the image
of $B_K(x)$ in $\Lam^{r+1}$ under the map $F$ will be contained in $B_\Lam(x)$, which is easier to work with.

The sifting sets are constructed by first choosing subsets
$\Om_\mfrak{p}\subset (\Cal{O}_K/\mfrak{p})^{r+1}$, and looking at
\[
 A_p=\vphi_p(\prod_{\mfrak{p}|p} A_\mfrak{p})\subset (\Cal{O}_K/p)^{r+1}
\]
where $A_\mfrak{p}$ is the complement of $\Om_\mfrak{p}$.  Here the map $\vphi_\mfrak{p}$ comes from the Chinese remainder theorem,
as discussed at the beginning of the section. Then we define $\Om_p$ to be the complement of $A_p$.  If $\nu_\mfrak{a}$
denotes the uniform probability measure on $(\Cal{O}_K/\mfrak{a})^{r+1}$ then the Chinese
remainder theorem then gives us
\begin{equation}\label{eq:CRT-1}
  1 - \nu_p(\Om_p)=\prod_{\mfrak{p}|p} (1-\nu_\mfrak{p}(\Om_\mfrak{p})).
\end{equation}
In order to estimate the remainder terms $R^\pm(X;Q^s)$, we will need the following lemma.
\begin{lma}\label{lma-lbs}
 Let $\mfrak{a}$ be an ideal, with $N(\mfrak{a})\leq D^d$.  If
 \[
 S_\mfrak{a}\subset (\Cal{O}_K/\mfrak{a})^{r+1}\quad\text{and}\quad S=B_\Lam(x)\cap\pi_\mfrak{a}^{-1}(S_\mfrak{a})
 \]
 then
\[
 \abs{P(S)-\frac{|S_\mfrak{a}|}{N(\mfrak{a})^{r+1}}}\ll_{K,r} \begin{cases}
								 \frac{D^2\log x}{x} &\text{if }d=r=1,\\
								 \frac{D^{d(r+1)}}{x} &\text{otherwise},
                                                              \end{cases}
\]
where the implied constant depends only on $K,r$.
 \end{lma}

\begin{proof}
Schanuel \cite{Schanuel} has proven that if $\mfrak{a}$ is an ideal of $\Cal{O}_K$,
then the number of lattice points in $B_\Lam(x)\cap (\mfrak{a}\Lam)^{r+1}$ is 
\[
 \varkappa \frac{x^{d(r+1)}}{N(\mfrak{a})^{r+1}}+\begin{cases}
						  O(x\log x)&\text{if }d=r=1,\\
                                                  O(x^{d(r+1)-1})&\text{otherwise},
                                                 \end{cases}
\]
where $\varkappa$ depends only on $K$ and on $d$, and the same goes for the implied constant in the error term (see also Serre \cite{Ser2}).
However, because of the method of proof, we can also use this estimate for the number of full copies of a system of representatives for
$(\Lam/\mfrak{a}\Lam)^{r+1}$ that can be found in $B_\Lam(x)$.  If $S\subset  (\Lam/\mfrak{a}\Lam)^{r+1}$, then this gives us the estimate
\[
 \frac{|S|}{|S_\mfrak{a}|}= \frac{|B_\Lam(x)|}{N(\mfrak{a})^{r+1}}+\begin{cases}
								    O\paren{\frac{|B_\Lam(x)|\log x}{x}}&\text{if }d=r=1,\\
                                                                    O\paren{\frac{|B_\Lam(x)|}{x}}&\text{otherwise}.
                                                                   \end{cases}
\]
If we multiply by $\frac{|S_\mfrak{a}|}{|B_\Lam(x)|}$ and use the bound $|S_\mfrak{a}|\leq D^{d(r+1)}$, then we obtain the result in the lemma.
\end{proof}

\begin{thm}\label{thm-lbs}
Given a homogeneous polynomial $f$ in $K[t_0,\ldots ,t_r]$, let $S$ be a finite subset of $\Sig_K^1$
\nomenclature{$\Sig_K^1$}{$\{\mfrak{p} \in \Sig_K:N(\mfrak{p})=q\text{ is prime }\}$} satisfying the following properties
\begin{enumerate}
 \item if $f \equiv 0\mod \mfrak{p}$, then $\mfrak{p}\in S$,
 \item if $\mfrak{p}\in S$ and $\mfrak{p}|p$, where $p$ is a prime in $\bb{Z}$, then $\mfrak{q}\in S$ for all $\mfrak{q}|p$,
 \item if $p$ ramifies in $K$ and $\mfrak{p}|p$ then $\mfrak{p}\in S$,
\end{enumerate}
and for $\mfrak{p}\notin S$ let
\begin{equation}\label{eq:lbs1}
 \Om_\mfrak{p}=\{(t_0,\ldots, t_r)\in\bb{F}_\mfrak{p}^{r+1}: f(t_0,\ldots t_r)\equiv 0 \mod \mfrak{p}\}.
\end{equation}
Let $\Cal{L}^*\subset \Sig_K^1(Q)\rcomp S$. \nomenclature{$\Sig_K^1(Q)$}{$\{\mfrak{p} \in \Sig_K(Q):N(\mfrak{p})=q\text{ is prime}\}$}
Then there exists $\kappa$ such that
\[
 |\{u\in B_K(x):\pi_\mfrak{p}(F(u))\notin \Om_\mfrak{p} \;\forall\mfrak{p}\in\Cal{L}^*\}|
\gg_{K,r} \frac{x^{r+1}}{(\log Q)^\kappa},
\]
so long as $Q^{s(d(r+1)+1)}\leq x^\frac{1}{2}$ and $s> 9\kappa+1 +10\cdot \log C$, where $C$ satisfies \prettyref{eq:14}.
\end{thm}

\begin{proof}
Let $W_\mfrak{p}$ denote the projective variety defined by $f=0$ over $\bb{F}_\mfrak{p}$, and let $q=|\bb{F}_\mfrak{p}|$.
Trivially, we have
\[
 |W_\mfrak{p}(\bb{F}_\mfrak{p})|\leq (\deg f) \frac{q^r-1}{q-1}.
\]
Since each point in $\bb{P}^r(\bb{F}_\mfrak{p})$ corresponds with $q-1$ points in $\bb{F}_\mfrak{p}^{r+1}$, we obtain
\begin{equation}\label{eq:11}
  \nu_\mfrak{p}(\Om_\mfrak{p})\leq \frac{\deg f}{N(\mfrak{p})}.
\end{equation}
Suppose that $\mfrak{p}\in S$ and lies over $p$.  Then by \prettyref{eq:CRT-1} we have
\[
 1-\nu_p(\Om_p)\geq\prod_{\mfrak{p}|p}\paren{1-\frac{\deg f}{|\bb{F}_\mfrak{p}|}}=\paren{1-\frac{\deg f}{p}}^d
\]
for $p>\deg f$, and it can be shown that
\[
 \prod_{w\leq p<Q}\paren{1-\frac{\deg f}{p}}^{-d}\ll\paren{\frac{\log Q}{\log w}}^{d\cdot \deg f}
\]
for $w>\deg f$, where the implied constant depends on the error term in the prime number theorem.  This is  sufficient
to obtain \prettyref{eq:14}, since we can adjust the constant to account for the finite number of factors with $p<\deg f$.
Since $|B_K(x)|\asymp |B_\Lam(x)|$ then by applying lemma \ref{lma-lbs} to $\mfrak{a}=d\Cal{O}_K$ and $S=\Om_d$, we obtain
\[
 |r_d(B_K(x);\alpha)|\ll_{K,r} \begin{cases}
				  x(\log x)D^2&\text{if }d=r=1,\\
				  x^r D^{d(r+1)}&\text{otherwise}.
                               \end{cases}
\]
For $D=Q^s$, this gives us
\[
 R^-(B_K(x);\alpha)\ll_{K,r} \begin{cases}
				  x(\log x )Q^{3s}&\text{if }d=r=1,\\
				  x^r Q^{s(d(r+1)+1)}&\text{otherwise},
                               \end{cases}
\]
which will be an acceptable error if $Q^{s(d(r+1)+1)}\leq x^\frac{1}{2}$.  The prime number theorem gives us
\[
 \prod_{p<Q}\paren{1-\frac{\kappa_1}{p}}\ll\frac{1}{(\log Q)^{d\deg f}},
\]
and so by applying theorem \ref{Iwan-Kow} with $\alpha(x)=1$, we obtain the result stated in the theorem with $\kappa=d\cdot \deg f$.
\end{proof}

\subsection{Application and an example}\label{subsec-eg}
As an application of the theorem just proven, consider the following corollary
\begin{cor} 
Let $V$ be an affine variety over $K$ birationally equivalent to $\bb{P}_K^r$, and let $A$ be an abelian variety defined
over the function field of $V$.  If we define
\[
 C(x)=\{t\in B_K(x): A_t\text{ has good reduction for all }\mfrak{p}\in\Cal{L}^*\}
\]
then there exist $s,\kappa$, such that
\[
 |C(x)|\gg_{K,r} \frac{x^{r+1}}{(\log x)^\kappa}.
\]
\end{cor}
Unfortunately, while the value of $\kappa$ sees to be sharp, the relationship between $x,Q$ as given by theorem \ref{thm-lbs} does not.
We illustrate this with the following example.

Let $V=\bb{A}^1_\bb{Q}$, and consider the elliptic curve
\[
 y^2=x^3+3(1-t)tx+2(1-t)^2t,
\]
over the function field $\bb{Q}(t)$. Then $\Delta(t)=-12^3(1-t)^3t^2$ and $j(t)=-12^3t$. Using asymptotic estimates for the Euler $\phi$-function, we find that
\[
 |B_K(x)|=\frac{12}{\pi^2}x^2+O(x\log x).
\]
If we homogenize $\Delta(t)$ by the substitution $t=\frac{t_1}{t_0}$, then we see that bad reduction can only occur at $p$ when $p=2,3$ or $p|t_1$ or $t_0-t_1$.
The $t_0,t_1$, are integers with absolute value $\leq x$, such that $t_1,t_0-t_1$ are not divisible by any prime $p\leq Q$. If we take $Q=x^\frac{1}{2}$,
the sieve of Eratosthenes shows that $|t_1|$ must be a prime  in the interval $(x^{\frac{1}{2}},x]$, and so it follows that 
\[
|C_K(x)|\gg \frac{x^2}{(\log x)^2}.
\]
On the other hand, applying \ref{thm-lbs} to the homogeneous polynomial $f(t_0,t_1)=6t_1(t_0-t_1)$ with $Q=x^\frac{1}{6s}$, we get this same
estimate except with the constant depending on $s$. But, since $s> 37 +10\cdot \log C$ and $C$ satisfies \prettyref{eq:14}, which in this case is a
consequence of the prime number theorem, this means that the value of $Q$ used in theorem \ref{thm-lbs} must be taken much smaller than $x^\frac{1}{2}$.

\section{General machinery}\label{sec2}
All definitions made in the previous section will be maintained, and we will supplement them with the following. For any group $G$,
the set of conjugacy classes will be denoted by $G^\#$.\nomenclature{$G^\#$}{the set of conjugacy classes of a group $G$}
If $V$ is a variety over a field $k$, then $\wbar{\eta}$ will denote a geometric point, and $\wbar{V}$ will denote the extension of
$V$ to a separable closure $k^\text{sep}$.

Now in particular if $f:U\to V$ is a finite \'etale covering over $k$ with arithmetic monodromy group $G$ and geometric monodromy group $G^g$,
then we get the following commutative diagram:
\begin{equation}\label{eq:15}
 \xymatrix{
	  &1\ar[d]					&1\ar[d]				&1\ar[d]			\\
  1\ar[r] & \pi_1(\wbar{U},\wbar{\xi})\ar[r]\ar[d]	& \pi_1(U,\wbar{\xi})\ar[r]\ar[d]	&\Gal(k^\sep /k)\ar[r]\ar[d]	&1\\
  1\ar[r] & \pi_1(\wbar{V},\wbar{\eta})\ar[r]\ar[d]	& \pi_1(V,\wbar{\eta})\ar[r]\ar[d]^\rho	&\Gal(k^\sep /k)\ar[r]\ar[d]	&1\\
  1\ar[r] & G^g\ar[r]\ar[d]				& G\ar[r]^m\ar[d]				& G/G^g\ar[r]\ar[d]		&1\\
	  &1						&1					&1
}
\end{equation}
The important point is that we must consider this diagram both where $k$ is a number field $K$,
and where $k$ is a finite field $\bb{F}_\mfrak{p}$, which we obtain by reducing mod $\mfrak{p}$
for some $\mfrak{p}\in \Sig_K$. Moreover, we will have a family of coverings $f_l:V_l\to V$ over $K$,
parametrized by $l$, where $l$ is a prime. In this case, we will denote the corresponding monodromy
groups by $G_l$\nomenclature{$G_l$}{the arithmetic monodromy group associated with the covering $V_l\to V$}
and $G_l^g$ respectively, and the map $\rho$ will be renamed $\rho_l$.  For each $l$, when we reduce
mod $\mfrak{p}$ to obtain the analogous situation over $\bb{F}_\mfrak{p}$, we will also introduce
$\mfrak{p}$ as a subscript: $G_{l,\mfrak{p}}$, $G_{l,\mfrak{p}}^g$
\nomenclature{$G_{l,\mfrak{p}}$}{the arithmetic monodromy group associated with the covering $V_{l,\mfrak{p}}\to V_\mfrak{p}$}
and $\rho_{l,\mfrak{p}}$. \nomenclature{$V_\mfrak{p}$}{the reduction of a variety $V$ mod $\mfrak{p}$}
The question arises, with $l$ fixed, when do we have $G_l\cong G_{l,\mfrak{p}}$ or similarly $G_l^g\cong G_{l,\mfrak{p}}^g$.
This may not happen for all $\mfrak{p}$, but in our application we manage to get sufficient control over this.

\subsection{The Chebotarev density theorem}
The following theorem is based on a version of the Chebotarev density theorem originally due to Lang \cite{Lang2},
which we sharpen with the help of a recent result by Kowalski \cite{Kow1}.  Since it applies to more general situations
than the one here, we will state it for a single Galois covering $f:U\to V$, and thus we will omit the subscripts $l,\mfrak{p}$.

\begin{thm}\label{thm-Lang}
Let $V$ be a smooth geometrically connected affine variety over $\bb{F}_q$ of dimension $r\geq 1$, and let $f: U\to  V$ be a finite
\'etale covering, with arithmetic Galois group $G$ such that $(|G|, q)=1$.  Let $\gamma$ be the image of $-1$
under the vertical map $\Gal(k^\sep /k)\to G/G^g$ in diagram \prettyref{eq:15}, where $k=\bb{F}_q$ in this case,
and let $C\in G^\#$ be arbitrary. \nomenclature{$C$}{a conjugacy class of a group $G$}  If we define
\[
 \Om_C(n) =\{t\in V(\bb{F}_{q^n}): \rho(\frob_t)\in C \},
\]
then \nomenclature{$\Om_C(n)$}{$\{t\in V(\bb{F}_{q^n}): \rho(\frob_t)\in C \}$}
\begin{equation}\label{eq:2}
 \abs{\frac{|\Om_C(n)|}{q^{nr}}-\frac{|C\cap m^{-1}(\gamma^n)|}{|G^g|}}\ll |G^g|^\frac{1}{2}|G^{g\,\#}|^\frac{1}{2}|C|q^{-\frac{n}{2}}
\end{equation}
where the implied constant depends only on $\wbar{V}$, and in particular it does not depend on $q$ or on $G$.
\end{thm}

\begin{proof}
Let $\Irr(G),\Irr(G/G^g)$ denote the irreducible representations of $G$ and $G/G^g$ respectively. Two
representations $\pi,\pi'\in \Irr(G)$ are equivalent when restricted to $G^g$  if there exists
$\psi \in \Irr(G/G^g)$ such that
\[
 \pi=\pi'\otimes(\psi\circ m).
\]
Let $\Pi$ denote a system of representatives with respect to this equivalence relation.  Then, if we sum over all
$\pi\in \Irr(G)$, we can break the sum up into a sum over $\psi\in \Irr(G/G^g)$ and a sum over $\pi\in \Pi$.
Note that $G/G^g$ is cyclic, and hence all irreducible representations of $G/G^g$ are characters. Using
these ideas together with orthogonality of characters gives us the following:
\begin{align*}
 |\Om_C||G|&=\sum_{g\in C}\sum_{t\in V(\bb{F}_{q^n})}\sum_{\pi\in \Irr(G)} \tr(\pi)(\rho(\frob_t))\wbar{\tr(\pi)(g)}\\
	   &=\sum_{g\in C}\sum_{t\in V(\bb{F}_{q^n})}\sum_{\pi\in \Pi}
	  \sum_{\psi\in \Irr(G/G^g)} \tr(\pi\otimes (\psi\circ m))(\rho(\frob_t))\wbar{\tr(\pi\otimes (\psi\circ m))(g)}\\
	   &=\sum_{\pi\in \Pi}\sum_{g\in C}\sum_{t\in V(\bb{F}_{q^n})} \tr(\pi)(\rho(\frob_t))\wbar{\tr(\pi)(g)}
				    \sum_{\psi\in \Irr(G/G^g)} \psi((m\circ \rho)(\frob_t))\wbar{\psi(g)}.
\end{align*}
Since the image of $\frob_t$ in $\Gal(\wbar{\bb{F}_q}/\bb{F}_q)$ is $-n$, by commutativity of the diagram \prettyref{eq:15}
we have $(m\circ\rho)(\frob_t)=\gamma^n$, and by using orthogonality of characters again, this means that
\[
 \sum_\psi \psi((m\circ \rho)(\frob_t))\wbar{\psi(g)}=\begin{cases}
                                                       |G/G^g| &\text{if } g\in C\cap m^{-1}(\gamma^n)\\
                                                        0      &\text{otherwise.}
                                                      \end{cases}
\]
By applying this to the summation above, we obtain
\[
 |\Om_C||G|=|V(\bb{F}_{q^n})||C\cap m^{-1}(\gamma^n)||G/G^g|
	     +\mathop{\sum_{\pi\in\Pi}}_{\pi\neq 1}\sum_{g\in C\cap m^{-1}(\gamma^n)}
	     \sum_{t\in V(\bb{F}_{q^n})} \tr(\pi)(\rho(\frob_t))\wbar{\tr(\pi)(g)},
\]
and hence we have the inequality
\[
 \abs{\frac{|\Om_C|}{|V(\bb{F}_{q^n})|}-\frac{|C\cap m^{-1}(\gamma^n)|}{|G^g|}}
 \leq \frac{|C|}{|G||V(\bb{F}_{q^n})|}\sum_{\pi\neq 1}\abs{\sum_{t\in V(\bb{F}_{q^n})} \tr(\pi)(\rho(\frob_t))}.
\]
If we apply proposition 5.1 (i) in \cite{Kow1} with $l=l',\pi'=1$, then we obtain
\[
\mathop{\sum_{\pi\in\Pi}}_{\pi\neq 1}\abs{\sum_{t\in V(\bb{F}_{q^n})}
\tr(\pi)(\rho(\frob_t))}\leq C(\wbar{V})q^{rn-\frac{1}{2}}|G|\mathop{\sum_{\pi\in\Pi}}_{\pi\neq 1}\dim \pi,  
\]
where $C(\wbar{V})$ is a constant depending only on $\wbar{V}$. Then by using Cauchy's inequality, 
we get the result in the statement of the theorem.
\end{proof}

\subsection{The general theorem}
Let $V$ be a geometrically irreducible affine variety over $K$ of dimension $r\geq 1$, so that it can be extended to a scheme 
over $\spec \Cal{O}_K[\frac{1}{\alpha}]$.  Also suppose that $V$ is birationally equivalent to $\bb{P}_K^r$
via the rational map $\vphi:V\to\bb{P}_K^r$, and let
\begin{equation}\label{eq:9}
  B_K(x)=\{t\in V(K):\exists \ang{U,\vphi}\text{ such that }t\in U(K)\text{ and }H_\vphi(t)\leq x\},
\end{equation}
where $\ang{U,\vphi}$ is a pair defining this rational map, and $H_\vphi$ denotes the corresponding absolute height on $U$.
For each prime $l$, let $V_l$ be an affine variety over $K$, such that
\smallskip

\begin{enumerate}
 \item $V_l$ can be extended to a scheme over $\spec\Cal{O}_K[\frac{1}{\alpha_l}]$ \label{cond:1}.
 \item $f_l:V_l\to V$ is finite \'etale over $\spec\Cal{O}_K[\frac{1}{\alpha_lg(x_1,\ldots, x_r)},x_1,\ldots, x_r]$.\label{cond:2}
 \item $G_l/G_l^g$ is isomorphic to a subgroup of $\bb{F}_l^\times$ where $G_l,G_l^g$ respectively denote the arithmetic and
       geometric Galois groups of the covering $V_l\to V$ when working over $K$.
 \item by extending scalars from $K$ to $K(\zeta_l)$, where $\zeta_l$ is a primitive $l$ root of unity,
       each connected component of $V_l$ is geometrically irreducible.\label{cond:4}
\end{enumerate}
\smallskip

Here, we introduce the function $g$ since it may not be possible to obtain the finite \'etale condition over
$\spec\Cal{O}_K[\frac{1}{\alpha_l},x_1,\ldots, x_r]$ itself.  This already happens in the case of elliptic curves,
as will be seen later. However, at least for our application it is not necessary to let $g$ vary with $l$, and since
causes the estimates to become considerably more complicated, we assume that $g$ is the same for all $l$.

\noindent We will need some control over $|G_l^g|$,$|G_l^{g\,\#}|$, and $\alpha_l$.  This is accomplished by the following properties.
\begin{prty}\label{prty1}
There exist constants $\beta_1,\beta_2$ such that
\[
 |G_l|\ll l^{\beta_1}\text{ and }|G_l^\#|\ll l^{\beta_2}
\]
where the implied constants are absolute.
\end{prty}
Obviously the condition $|G_l^\#|\ll l^{\beta_2}$ is not necessary in property \ref{prty1}, however,
in our application it is possible to take $\beta_2$ smaller than $\beta_1$, giving us a sharper bound.

\begin{prty}\label{prty3}
There exists a constant $\beta_3$, such that
\[
 |\{\mfrak{p}\in \Sig_K:\alpha_l\in\mfrak{p}\text{ or }G_{l,\mfrak{p}}^g\ncong G_l^g\}|\ll l^{\beta_3}
\]
holds for all $l$, where the implied constant is absolute.
\end{prty}
Since $V$ does not depend on $l$, then by working over $\spec{O}_K[\frac{1}{\alpha}]$ theorem 9.7.7 of \cite{EGA4_3} shows that there are only a
finite number of $\mfrak{p}$ for which $V_\mfrak{p}$ is not geometrically irreducible, but it is too soft to be of much help with $V_l$.  Instead, we rely
on some known facts about moduli spaces of abelian varieties to get this property to hold with $\beta_3=1$. 

If $t\in V(K)$, then $G_K=\Gal(\wbar{K}/K)$ \nomenclature{$G_K$}{$\Gal(\wbar{K}/K)$} acts on the fiber $V_{l,t}$,\nomenclature{$V_{l,t}$}{the fiber of $V_l$ over $t$}
which gives us a homomorphism $\rho_{l,t}: G_K \to G_{l,t}$.  To determine the $t$ for which the image contains $G_l^g$, we use the following lemma of Jordan (see \cite{Ser3}):
\begin{lma}[Jordan]\label{lmaB}
 Given a finite group $G$ and a subgroup $H\subset G$, if $H\cap C \neq\emptyset$ for every $C\in G^\#$, then $H=G$.
\end{lma}
In light of this lemma, the strategy is as follows. For each $l$, and each conjugacy class, we apply theorem \ref{thm-sieve} with $\Cal{L}^*$ taken to be a suitable subset
of $\Sig_K^1(Q;1,l)$. This definition accomplishes a couple things.  First, it means that constant in the large sieve does not depend on $l$ because
the sieve setting and $\vphi$ remain the same for all $l$.  Second, by Artin reciprocity we know that if $\mfrak{p}\in \Cal{L}^*$, then $\mfrak{p}$ splits over $K(\zeta_l)$,
and hence if $\mfrak{P}$ is a prime above $\mfrak{p}$ in $\Sig_{K(\zeta_l)}$, then regardless of whether we reduce $V_l\to V$ mod $\mfrak{p}$ or mod $\mfrak{P}$,
we end up with a covering defined over $\bb{F}_p$.  This makes it possible to define the sets $\Om_\mfrak{p}$. However, a side effect is that it is impossible to
obtain information about $G_l$ itself by working over $\bb{F}_p$, so in particular this forces us to apply
Jordan's lemma to $G_l^g$.

In our application, we will also need an effective version of Hilbert irreducibility for a single value of $l$, so we state it in the following lemma:
\begin{lma}[Hilbert irreducibility]
\label{H-Irred-lma}
 Let $K$ be a number field, and let $V$ be a smooth geometrically connected affine variety over $K$, birationally equivalent to
 $\bb{P}_K^r$ for some $r\geq 1$. Fix a prime $l$, and let $V_l\to V$ be a covering satisfying conditions \ref{cond:1}-\ref{cond:4} above.
 If we define
 \nomenclature{$E_{K,l}(x)$}{$\{t\in B_K(x):G_l^g\not\subset\rho_{l,t}(G_K)\}$} 
\[
 E_{K,l}(x)=\{t\in B_K(x):G_l^g\not\subset\rho_{l,t}(G_K)\},
\]
then for sufficiently large $x$ depending on $l$, we have
\begin{equation}\label{eq:13}
  \frac{|E_{K,l}(x)|}{|B_K(x)|}\ll_{\vphi,K,r}|G_l^g||G_l^{g\,\#}|\cdot l\frac{\log x}{x^\frac{1}{2}}.
\end{equation}
\end{lma}

\begin{rmk*}
Following the same argument used for Serre's Proposition 19 in \cite{Ser1}, we can get surjectivity of the map
$\rho_{l,t}:G_K\to G_l$ from $G_l^g\subset\rho_{l,t}(G_K)$ together with surjective of the map $m$ in diagram \prettyref{eq:15}.
As a result, it is in fact true that
\[
 E_{K,l}(x)=\{t\in B_K(x):\rho_{l,t}(G_K)\subsetneq G_l\}.
\]
It is also important to realize that the constant $C(\wbar{V})$ from the Chebotarev Density Theorem for function fields does not
affect the constant in \prettyref{eq:13}, rather it affects what we mean by ``sufficiently large $x$. This is accomplished by the
definition of $\Cal{L}^*$.  If we were to let $g$ vary with $l$ in condition \ref{cond:2} above, then $C(\wbar{V})$ would also
vary with $l$, and we would need to assume subexponential growth in both cases to get anything useful.
\end{rmk*}

\begin{proof}
Since the absolute height is used in the definition of $B_K(x)$ we use the adjustment to the large sieve made
in the remark just below theorem \ref{thm-sieve}.  Let $U_l$ be a connected component of $V_l$ over $K(\zeta_l)$,
so that it is geometrically irreducible, and let
\[
 \Cal{L}^*=\{\mfrak{p}\in\Sigma_K^1(Q;1,l): \alpha_l\notin \mfrak{p},\,
						      N(\mfrak{p})\gg \max\{|G_l^g|^3|G_l^{g\,\#}|,|G_l^g|\},
						      \text{ and }G_{l,\mfrak{p}}^g\cong G_l^g\}.
\]
If $\mfrak{p}\in\Cal{L}^*$ and $\mfrak{P}$ is a prime over $\mfrak{p}$ in $\Sig_{K(\zeta_l)}$, reducing mod $\mfrak{P}$ gives us a covering
$U_{l,\mfrak{p}}\to V_\mfrak{p}$ with Galois group isomorphic to $G_l^g$.  In particular this means both $U_l,V$ remain geometrically irreducible.
Since $V_l\to V$ is \'etale over $\spec\Cal{O}_K[\frac{1}{\alpha_lg(x_1,\ldots, x_r)},x_1,\ldots, x_r]$, then we get an \'etale
covering over $\bb{F}_p$ by removing the points satisfying the homogeneous polynomial equation $g(t_0,t_1,\ldots t_r)\equiv 0 \mod p$, obtained
by the substitution $x_1=\frac{t_1}{t_0},\ldots x_r=\frac{t_r}{t_0}$.  The inequality \prettyref{eq:11} gives us an estimate for the number of
$\bb{F}_p$-points being removed to obtain an \'etale covering, and once this has been done, theorem \ref{thm-Lang} applies.

So, now let $C\in G_l^{g\,\#}$ be arbitrary, and for all $\mfrak{p}\in \Cal{L}^*$ define
\[
 \Omega_{\mfrak{p},C}=\{t\in V(\bb{F}_\mfrak{p}): \rho_{l,\mfrak{p}}(\frob_t)\in C \}.
\]
Strictly speaking this is not a subset of $(\Lam/\mfrak{p}\Lam)^{r+1}$, but by first mapping into $\bb{P}^r$ with the help of $\vphi$, we
can then map the points of $\Omega_{\mfrak{p},C}$ directly into $(\Lam/\mfrak{p}\Lam)^{r+1}$.  Since we only need a lower bound, it is
sufficient to obtain one for the set just defined.

We apply the large sieve with this data, and we want to estimate the right hand side of \prettyref{eq:4} in a useful way,
with the help of the Chebotarev density theorem. Specifically, using \prettyref{eq:2}, we get the following result 
\begin{equation}\label{eq:1}
 L(Q)\geq \sum_{\mfrak{p}\in \Cal{L}^*}\nu_\mfrak{p}(\Om_{\mfrak{p},C})\gg \sum_{\mfrak{p}\in\Cal{L}^*}\frac{|C|}{|G_l|}.
\end{equation}
The implied constant does not depend on $l$, in fact we can take the constant to be $\frac{1}{2}$, which is possible if we take the 
constant in $N(\mfrak{p})\gg \max\{|G_l^g|^3|G_l^{g\,\#}|,|G_l^g|\}$ to be $\max\{(4C(\wbar{V}))^2,4\}$,
where $C(\wbar{V})$ is the constant in \prettyref{eq:2}.

We now must obtain an underestimate for $|\Cal{L}^*|$.  Of course $\Cal{L}^*$ could be empty, but by choosing $Q$ and $x$
sufficiently large this will not happen.  However, since the definition of $\Cal{L}^*$ depends on $l$, what we mean by sufficiently
large will also depend on $l$, even though the implied constant we get will not depend on $l$. In the context of this lemma $l$ is fixed,
so Grothendieck's theorem 9.7.7 of \cite{EGA4_3} shows that the set of $\mfrak{p}\in\Sig_K$ for which $G_{l,\mfrak{p}}^g\ncong G_l^g$ is finite,
and hence the difference
\[
 |\Sig_K^1(Q;1,l)|-|\Cal{L}^*|
\]
is bounded.  This means that we can use Siegel-Walfisz theorem to get an estimate for $|\Cal{L}^*|$. Specifically, since $Q=x^\frac{1}{2}$,
then we have
\[
 |\Cal{L}^*|\gg_K \frac{x^\frac{1}{2}}{l\cdot\log x},
\]
for $x$ large enough that $\Cal{L}^*$ is non-empty. By applying this to \prettyref{eq:1} we get the following
underestimate for $L(Q)$,
\[
 L(Q)\gg_K \frac{|C|}{|G_l^g|}\frac{x^\frac{1}{2}}{l\cdot \log x},
\]
and then using this estimate of $L(Q)$ together with \prettyref{eq:3} gives us the following upper bound for the sifted set
\begin{equation}\label{eq:8}
 |S(B_K(x),(\Om_{\mfrak{p},C});\Cal{L}^*)|\ll_{\vphi,K,r} \frac{|G_l^g|}{|C|}\cdot l \frac{\log x}{x^\frac{1}{2}}x^{r+1}.
\end{equation}
In the future we will use the following shorthand for the sifted set, to make the remaining argument easier to follow
\[
 Y_C(x)=S(B_K(x),(\Om_{\mfrak{p},C});\Cal{L}^*).
\]
Since $V$ is a rational variety, we have $B_K(x)\asymp x^{r+1}$, and so by \prettyref{eq:8}
\begin{equation}\label{eq:7}
 \frac{|Y_C(x)|}{|B_K(x)|}\ll_{\vphi,K,r} \frac{|G_l^g|}{|C|}\cdot l \frac{\log x}{x^\frac{1}{2}}
\end{equation}
for any conjugacy class $C$ of $G_l^g$. Now by lemma \ref{lmaB} $E_{K,l}(x)\subset \bigcup_{C\in G_l^{g\,\#}} Y_C (x)$, so
\begin{equation}\label{eq:6}
 \frac{|E_{K,l}(x)|}{|B_K(x)|}\leq \sum_{C\in G_l^{g\,\#}} \frac{|Y_C(x)|}{|B_K(x)|}.
\end{equation}
By combining these last two estimates, and using the trivial bound $|C|\geq 1$ for each conjugacy class, we obtain the estimate
given in the lemma.
\end{proof}

\begin{rmk*}
Grothendieck's theorem 9.7.7 is proven in such a soft way, that it is not clear if it gives any control over the size of the finite set
of bad primes as $l$ varies.  However, if this set grows only as a power of $l$, then $x$ can be chosen sufficiently large so that the
lemma works simultaneously for all $l$  less than some power of $\log x$.  This is the idea behind property~\ref{prty3}.
For an effective result, we can adjust the argument using the pigeon hole principle in a way completely analogous to Zywina's
version for elliptic curves \cite{Zyw1}.
\end{rmk*}

To prove a similar result that holds for all $l$, requires stronger methods.  If we define
\[
 E_K(x)= \bigcup_{\forall l} E_{K,l}(x),
\]
\nomenclature{$E_K(x)$}{$C_K(x)\cap\bigcup_{\forall l} E_{K,l}(x)$} 
then we need the following property to hold:
\begin{prty}\label{prty4}
There exists $C_K(x)\subset B_K(x)$ and constant $\gamma$ s.t.
\begin{align*}
 &\lim_{x\to\infty} \frac{|B_K(x)\rcomp C_K(x)|}{|B_K(x)|}=0\text{ and,}\\
 C_K(x)\cap E_K(x) \subset \bigcup_{l\ll (\log x)^\gamma} E_{K,l}(x).
\end{align*}
\end{prty}

Now, here is the precise statement of our general result.
\begin{thm}\label{general-thm}
 Let $K$ be a number field, and let $V$ be a smooth geometrically connected affine variety over $K$, birationally equivalent to
 $\bb{P}_K^r$ for some $r\geq 1$. Let $(V_l\to V)_l$ be a family of coverings satisfying conditions \ref{cond:1}-\ref{cond:4} above, and
 such that properties \ref{prty1}-\ref{prty4} hold with the values $\beta_1,\beta_2,\beta_3,\gamma$. Then for sufficiently large $x$, we have
\[
 \frac{|E_K(x)|}{|B_K(x)|}\ll_{\vphi,K,r} \frac{|B_K(x)\rcomp C_K(x)|}{|B_K(x)|}+\frac{(\log x)^{(\beta_1+\beta_2+2)\gamma +1}}{x^\frac{1}{2}}.
\]
\end{thm}

\begin{proof}
The proof begins in the same way as it does for lemma \ref{H-Irred-lma}, except that we now need properties \ref{prty1}-\ref{prty4}
in order to get the estimate for $|\Cal{L}^*|$.  Specifically, by applying these properties to the definition of $\Cal{L}^*$,, we obtain
\[
 |\Sig_K^1(Q;1,l)|-|\Cal{L}^*|\ll (\log x)^{\max\{3\beta_1+\beta_2,\beta_1,\beta_3\}\gamma},
\]
where the implied constant is absolute.  Comparing this with the estimate that Siegel-Walfisz gives us for $|\Sig_K^1(Q;1,l)|$ shows that
this is indeed smaller for sufficiently large $x$.  Continuing with the rest of the proof of lemma \ref{H-Irred-lma}, we obtain
\prettyref{eq:13} just as before, except that what we mean by sufficiently large $x$ no longer depends on $l$. We can now rewrite it
purely in terms of $x$ and $l$ by using the bounds in property \ref{prty1}.  Specifically, we get
\begin{equation}\label{eq:5}
 \frac{|E_{K,l}(x)|}{|B_K(x)|}\ll_{\vphi,K,r}l^{\beta_1+\beta_2+1} \frac{\log x}{x^\frac{1}{2}}.
\end{equation}
Now by property \ref{prty4} we have
\[
 \frac{|E_K(x)|}{|B_K(x)|}\leq \frac{|B_K(x)\rcomp C_K(x)|}{|B_K(x)|} +\sum_{l\ll (\log x)^\gamma}\frac{|E_{K,l}(x)|}{|B_K(x)|},
\]
and by applying \prettyref{eq:5} to the summation we have
\[
\sum_{l\ll (\log x)^\gamma}\frac{|E_{K,l}(x)|}{|B_K(x)|}
\ll_{\vphi,K,r} \frac{(\log x)^{(\beta_1+\beta_2+2)\gamma+1}}{x^\frac{1}{2}}.
\]
The last two estimates together give us the result in the theorem.
\end{proof}

\section{Application to abelian varieties}\label{sec4}

We first recall some important facts about the Siegel moduli spaces $X_N$ of level $N$ (see Chap. IV sec. 6 in \cite{Faltings}).
Strictly speaking they are locally noetherian separated algebraic stacks over $\spec\bb{Z}[\frac{1}{N}]$.  For $M|N$, the natural
morphisms $X_N\to X_M$ are finite \'etale over $\spec \bb{Z}[\zeta_N,\frac{1}{N}]$, and by taking the toroidal compactification
we obtain geometric irreducibility over $\spec\bb{Z}[\zeta_N,\frac{1}{N}]$ (6.8 Corollary 1 in \cite{Faltings}).  Considering the
same picture over $\spec\bb{Z}[\frac{1}{N}]$ is possible, but geometric irreducibility no longer holds.  This is also explained
in Sections 1.4 and 1.10 of \cite{Katz1} for the special case of elliptic curves.

Now, let $K$ be a number field, let $V$ be a smooth geometrically irreducible affine variety over $K$ of dimension $r\geq 1$,
birationally equivalent to $\bb{P}_K^r$ via the rational map $\vphi:V\to\bb{P}_K^r$, and let $x_1,\ldots x_r$ be a generating
set of global functions on $V$.  Let $K(V)$ be the function field of $V$, and let $A$ be a principally polarized abelian variety
over $K(V)$ of dimension $g$, that is non-constant for any of the global sections $x_1,\ldots x_r$. We can spread out $A$ to an 
abelian scheme over $U$, where $U$ is an open subset of $V$.  For each $t\in U$, $A_t$ corresponds to a point in $X=X_1$,
the Siegel moduli space of level 1.  This defines a map $\psi:V\to X$, which is dominant.  We then define $V_l=V\times_X X_l$,
so that we get the following commutative diagram.
\[
\xymatrix{V_l\ar[d]\ar[r] & X_l\ar[d]\\ V \ar[r]^\psi& X}
\]
The map $V_l\to V$ itself may not be finite \'etale, but we can get a finite \'etale map by restricting to $U$.  Using the theory
of stacks this is accomplished by the fact that $X_n\to X$ is relatively representable (see Chapter I section 4 of \cite{Faltings}
for the definition).  Using schemes, it must be shown that $V_l\to V$ is generically \'etale, not only in characteristic zero,
but also after reduction mod $\mfrak{p}$.  The stacks $X_n$ can only be conisdered as schemes for $n\geq 3$.  However, using the
method in \cite{Katz1} for constructing a scheme for the level 2 structure, we replace $X_n$ by the quotient of $X_{3n}$ by
the subgroup of matrices congruent to the identity mod 3.  For $n=1$ this quotient will replace $X$ and is an irreducible scheme.
Since $V$ and $X$ are irreducible and $V\to X$ is dominant, we get an extension of function fields, and a field theory argument
can be used to show that $V_l\to V$ is generically \'etale.

It is worthwile to consider an example, in the elliptic curve case. If $E$ is the elliptic curve
\[
y^2=x^3+3(1-t)tx+2(1-t)^2t
\]
defined over the function field $\bb{Q}(t)$, it can be spread out to an elliptic curve over the open subset of $\spec\bb{Z}[t]$
with $6t(1-t)$ inverted.  Since $j(t)=-12^3t$, this corresponds to the open subset of the $j$-line with $j$ and $1728-j$ inverted.
The scheme $V_l$ constructed as above, is then finite \'etale over $\spec\bb{Z}[\frac{1}{6lt(1-t)},t]$.  This means that the family
of coverings $V_l\to V$ will satisfy conditions \ref{cond:1}-\ref{cond:4} with $\alpha_l=6l$ and $g(t)=t(1-t)$.  Note
that $g$ does not depend on $l$.  The Galois groups in this case are $G_l=\GL{2}{l}/\ang{\pm 1}$ and $G_l^g=\SL{2}{l}/\ang{\pm 1}$,
and more generally for abelian varieties of dimension $g$ we will have $G_l=\GSp{2g}{l}/\ang{\pm 1}$ and $G_l^g=\Sp{2g}{l}/\ang{\pm 1}$,
at least so long as $K$ and $\bb{Q}(\zeta_l)$ are linearly disjoint.

In order to get geometric irreducibility to hold for the coverings $V_l \to V$ we make the assumption that $\psi:V\to X$ is a
dominant map with degree 1.  This is not possible for all Siegel moduli spaces.  It is known that the level 1 moduli spaces of
principally polarized abelian varieties are unirational when $g\leq 5$, that they are of general type when $g\geq 7$, and in the
case $g=6$  it seems that the question of unirationality is still open (see \cite{Mum}). It is also known that when the level is
$\geq 4$, moduli spaces of principally polarized abelian surfaces are of general type (see \cite{wang}).  In the case of abelian
surfaces it is known that the level 1 moduli space of polarized abelian surfaces is rational or unirational if the degree of
polarization is 1,2,3,4,5,7, or 9 (see \cite{Grit}), in particular for the principally polarized case it is actually rational
(see \cite{Igu2}). In the case of elliptic curves the analogous result is classical.

Let $G_K=\Gal(\wbar{K}/K)$, fix $t\in V(K)$ and consider the representation $\rho_{l,t}:G_K\to \text{GL}(A_t[l])$, where $A_t[l]$ denotes the $l$-torsion of $A_t$.
This representation induces a representation $\rho_{l,t}:G_K\to G_l$.  However, whereas $\text{GL}(A_t[l])\cong \GSp{2g}{l}$, we have $G_l\cong\GSp{2g}{l}/\ang{\pm 1}$
as noted above.  As a consequence of this, it is easy to pass results from $\text{GL}(A_t[l])$ to $G_l$, but to go backwards we will need the following lemma.
\begin{lma}\label{lmaC}
Given the exact sequence,
\begin{equation}\label{eq:ses1}
\xymatrix{1\ar[r] & \ang{\pm 1}\ar[r] &\GSp{2g}{l}\ar[r]^{\vphi\quad} &\GSp{2g}{l}/\ang{\pm1}\ar[r] &1}
\end{equation}
let $G$ be a subgroup of $\GSp{2g}{l}$ s.t. $\vphi|_G$ is surjective. Then $G=\GSp{2g}{l}$.
\end{lma}

\begin{proof}
It suffices to show that $-1\in G$. If $x\in \GSp{2g}{l}$, then $x$ or $-x\in G$.  In particular
\[
 \begin{pmatrix}
  &-I_g\\
I_g&
 \end{pmatrix}
\quad\text{or}\quad
\begin{pmatrix}
  &I_g\\
-I_g&
\end{pmatrix}
\]
is in $G$.  But by squaring both of these matrices, it follows that $-1\in G$.
\end{proof}

It is known that CM elliptic curves cannot have surjective Galois representations for all $l$, and for abelian
varieties it is true more generally if the endomorphism ring is bigger than $\bb{Z}$.  Hence, we define
\begin{equation}\label{eq:12}
 C_K(x)=\{t\in B_K(x):\End_{\wbar{K}}(A_t)=\bb{Z}\},
\end{equation}
and wish to show that the number of points removed from $B_K(x)$ to get $C_K(x)$ is small. It would be nice to say that we
simply need to remove the PEL varieties of lower dimension from the full Siegel moduli space, however, in general it is
not clear that we have a finite list of possible endomorphism rings, once $K$ has been fixed.

So instead, we fix $l$, and apply the Hilbert irreducibility lemma \ref{H-Irred-lma} to obtain a subset of $B_K(x)$,
for which $\rho_{l,t}(G_K)$ is all of $G_l$. We then apply lemma \ref{lmaC} to lift to $\GSp{2g}{l}$, and a result of
Vasiu \cite{Vasiu} to lift to $\GSp{2g}{\bb{Z}_l}$. This says that the number of $t\in B_K(x)$ for which the Galois
representation on the Tate module associated to $A_t$ is not surjective is bounded by a constant times
$|B_K(x)|\frac{\log x}{x^\frac{1}{2}}$. Faltings \cite{Faltings2} has shown that
\[
 \End_K(A_t)\otimes \bb{Z}_l\cong \End_{G_K}(T_lA_t)
\]
which means that if $\End_K(A_t)\neq\bb{Z}$, then $\End_{G_K}(T_lA_t)$ contains a non-central element, and so the Galois representation
of $G_K$ on $T_lA_t$ cannot be surjective.  It follows that
\begin{equation}\label{eq:16}
  \frac{|B_K(x)\rcomp C_K(x)|}{|B_K(x)|}\ll\frac{\log x}{x^\frac{1}{2}}.
\end{equation}

To show that property \ref{prty4} holds using this data, we use the theorem of Masser and W\"{u}sthol
\cite{MW} in the case $g=1$, and a generalization of it
by Kawamura \cite{Kawa} in the case $g=2$.

\begin{thm}[Kawamura]\label{thm-Kawamura}
Let $A$ be a principally polarized abelian surface over a number field of degree $d$ with $\End_{\wbar{K}}(A)\cong\bb{Z}$. 
Let $D(K)$ be the discriminant of $K$,
and $h(A)$ be the Faltings height of $A$. Then there exist constants $c,\gamma$, such that for any prime $l$ satisfying
\[
 l>\max\{D(K), c(\max\{d,h(A)\})^\gamma\},
\]
we have $\rho_l(G_K)=\GSp{4}{l}$
\end{thm}

Kawamura tries to use the Main Theorem in \cite{KL}, however, that theorem applies only when the dimension is $> 12$.
This condition is only needed for maximality, and so by not making any reference to maximality, Aschbacher's theorem  can be
applied instead (see Theorem 1.2.1 in \cite{KL}). There is another mistake in Kawamura's proof, specifically he claims that
tables \linebreak 3.5.A-H in \cite{KL} indicate that $\Cal{S}$ is empty, which is not true. This is also not a huge problem because the
groups in $\Cal{S}$ can be dealt with in the same way that he deals with $2^{1+4}.\On{4}{-}{2}$. With these minor changes,
the proof of Theorem \ref{thm-Kawamura} is valid.

\begin{lma}\label{lmaD}
 Property \ref{prty4} holds for the groups $G_l$ for $C_K(x)$ as defined by equation \prettyref{eq:12} in the cases $g=1$ and $2$.
\end{lma}

\begin{proof}
Let $t\in C_K(x)$ be a point for which $\rho_{l,t}:G_K\to G_l$ is not surjective, so that $E_{K,l}$ is non-empty.
Then the representation $\rho_{l,t}:G_K\to \text{GL}(A_t[l])$ is also not surjective, and so by Kawamura's theorem
there exist constants $c,\gamma$ such that
\[
 l\leq \max\{D(K), c(\max\{d,h(A_t)\})^\gamma\} \ll h(A_t)^\gamma.
\]
It is well known (see \cite{Silv} for example)from the theory of heights that $h(A_t)=h(t)+O(1)$, where the height on the right is the
absolute logarithmic height on $\bb{P}_{\wbar{\bb{Q}}}^r$, i.e. $h(t)=\log H(t)$.
The theory of height also gives us $H_\vphi(t)= H(t)+O(1)$, so it follows that
\[
 l \ll (\log H_\vphi(t))^\gamma\leq (\log x)^\gamma,
\]
and therefore $C_K(x)\cap E_K(x)$ will be contained in the finite union $\bigcup_{l\ll (\log x)^\gamma }E_{K,l}$.
\end{proof}

We are now in a position to prove the following theorem
\begin{thm}
Let $K$ be a number field, let $V$ be smooth geometrically irreducible affine variety over $K$ of dimension $r=\binom{g+1}{2}$,
birationally equivalent to $\bb{P}_K^r$ via the rational map $\vphi:V\to\bb{P}_K^r$.  Let $K(V)$ be the function field of $V$,
and let $A$ be a principally polarized abelian variety over $K(V)$ of dimension $g=1$ or 2.
\[
 E'_K(x)=\{t\in B_K(x): \rho_{l,t}(G_K)\subsetneq \GSp{2g}{l}\text{ for some $l$}\}.
\]
Then
\[
 \lim_{x\to \infty}\frac{|E'_K(x)|}{|B_K(x)|}=0.
\]
\end{thm}

\begin{proof}
In the cases $g=1$ and $2$, we know by lemma~\ref{lmaD} that property~\ref{prty4} holds with $C_K(x)$ as defined by equation~\prettyref{eq:12},
and that $X$ is a rational variety, hence a degree 1 map $\psi:V\to X$ exists. Therefore
\begin{align*}
 G_l^g &\cong\Sp{2g}{l}/\ang{\pm 1}\text{ for all $l$, and }\\
   G_l &\cong\GSp{2g}{l}/\ang{\pm 1}\text{ for all $l$ such that $K$ and $\bb{Q}(\zeta_l)$ are linearly disjoint.}
\end{align*}
By 6.8 corollary 1 in \cite{Faltings}, it is known that the level $l$ Siegel moduli spaces have irreducible geometric fibers over $\spec\bb{Z}[\zeta_l,\frac{1}{l}]$,
and this is sufficient to obtain $G_{l,\mfrak{p}}^g\cong G_l^g$ for primes that split over $K(\zeta_l)$. We have already pointed out that geometric irreducibility
of $V$ can be handled with theorem 9.7.7 of \cite{EGA4_3}. Taken together it follows that property~\ref{prty3} applies with $\beta_3=1$.

The discussion at the beginning of this section shows that if $U$ is the open subset of $V$ on which the
spreading out of $A$ is defined, then $f_l$ only needs to cover the complement of the intersection of $U$ and the domain of $\psi$.

For arbitrary dimension $g$ the order formulas in \cite{KL} together with the upper bounds for $|G_l^\#|$ in \cite{Lieb} (see also \cite{Kow1}),
show that property~\ref{prty1} holds with $\beta_1=2g^2+g+1$ and $\beta_2=g+1$, so by applying theorem~\ref{general-thm}
with these values, we obtain
\begin{equation}\label{eq:10}
 \frac{|E_K(x)|}{|B_K(x)|}\ll_{\vphi,K,r}  \frac{|B_K(x)\rcomp C_K(x)|}{|B_K(x)|} + \frac{(\log x)^{2(g^2+g+2)\gamma+1}}{x^\frac{1}{2}}.
\end{equation}
In particular if $g=1$ or $g=2$ the factor in front of $\gamma$ is 8 or 16 respectively, and the estimate given by \prettyref{eq:16} shows that
the first term can be dropped completely.

If we view $\rho_{l,t}$ as a representation on the $l$-torsion of $A_t$ then $\rho_{l,t}(G_K)\subset \GSp{4}{l}$.
But if we consider the induced map $\rho_{l,t}:G_K\to G_l$, then $\rho_{l,t}(G_K)\subset \GSp{4}{l}/\ang{\pm1}$.
The two images are related by the exact sequence \prettyref{eq:ses1}, hence $t\in E'_K(x)$ implies that $t\in E_K(x)$
by Lemma \ref{lmaC}. It follows that we can replace $E_K(x)$ with $E'_K(x)$ in the estimates above, and so taking
the limit $x\to \infty$ proves the theorem.
\end{proof}

\appendix
\section{List of notation}\label{notation}
\printnomenclature[2cm]

\bibliography{Larsen-Wallace}{}

\begin{thebibliography}{10}

\bibitem{Duke}
William Duke.
\newblock Elliptic curves with no exceptional primes.
\newblock {\em C. R. Acad. Sci. Paris S\'er. I Math.}, 325(8):813--818, 1997.

\bibitem{Faltings2}
Gerd Faltings.
\newblock Finiteness theorems for abelian varieties over number fields.
\newblock In {\em Arithmetic geometry ({S}torrs, {C}onn., 1984)}, pages 9--27.
  Springer, New York, 1986.
\newblock Translated from the German original [Invent. Math. {{\bf{7}}3}
  (1983), no. 3, 349--366; ibid. {{\bf{7}}5} (1984), no. 2, 381; MR
  85g:11026ab] by Edward Shipz.

\bibitem{Faltings}
Gerd Faltings and Ching-Li Chai.
\newblock {\em Degeneration of abelian varieties}, volume~22 of {\em Ergebnisse
  der Mathematik und ihrer Grenzgebiete (3) [Results in Mathematics and Related
  Areas (3)]}.
\newblock Springer-Verlag, Berlin, 1990.
\newblock With an appendix by David Mumford.

\bibitem{Grit}
Valeri Gritsenko.
\newblock Irrationality of the moduli spaces of polarized abelian surfaces.
\newblock In {\em Abelian varieties ({E}gloffstein, 1993)}, pages 63--84. de
  Gruyter, Berlin, 1995.
\newblock With an appendix by the author and K. Hulek.

\bibitem{EGA4_3}
A.~Grothendieck.
\newblock \'{E}l\'ements de g\'eom\'etrie alg\'ebrique. {IV}. \'{E}tude locale
  des sch\'emas et des morphismes de sch\'emas. {III}.
\newblock {\em Inst. Hautes \'Etudes Sci. Publ. Math.}, (28):255, 1966.

\bibitem{Igu2}
Jun-ichi Igusa.
\newblock On {S}iegel modular forms of genus two.
\newblock {\em Amer. J. Math.}, 84:175--200, 1962.

\bibitem{Jon2}
Nathan Jones.
\newblock Trace formulas and class number sums.
\newblock {\em Acta Arith.}, 132(4):301--313, 2008.

\bibitem{Kawa}
Takashi Kawamura.
\newblock The effective surjectivity of mod {$l$} {G}alois representations of
  1- and 2-dimensional abelian varieties with trivial endomorphism ring.
\newblock {\em Comment. Math. Helv.}, 78(3):486--493, 2003.

\bibitem{KL}
Peter Kleidman and Martin Liebeck.
\newblock {\em The subgroup structure of the finite classical groups}, volume
  129 of {\em London Mathematical Society Lecture Note Series}.
\newblock Cambridge University Press, Cambridge, 1990.

\bibitem{Kow1}
E.~Kowalski.
\newblock The large sieve, monodromy and zeta functions of curves.
\newblock {\em J. Reine Angew. Math.}, 601:29--69, 2006.

\bibitem{Kow2}
E.~Kowalski.
\newblock {\em The large sieve and its applications}, volume 175 of {\em
  Cambridge Tracts in Mathematics}.
\newblock Cambridge University Press, Cambridge, 2008.
\newblock Arithmetic geometry, random walks and discrete groups.

\bibitem{Katz1}
Willem Kuyk and J.-P. Serre, editors.
\newblock {\em Modular functions of one variable. {III}}.
\newblock Lecture Notes in Mathematics, Vol. 350. Springer-Verlag, Berlin,
  1973.

\bibitem{Lang2}
Serge Lang.
\newblock Sur les s\'eries {$L$} d'une vari\'et\'e alg\'ebrique.
\newblock {\em Bull. Soc. Math. France}, 84:385--407, 1956.

\bibitem{Lang1}
Serge Lang and Andr{\'e} Weil.
\newblock Number of points of varieties in finite fields.
\newblock {\em Amer. J. Math.}, 76:819--827, 1954.

\bibitem{Lieb}
Martin~W. Liebeck and L{\'a}szl{\'o} Pyber.
\newblock Upper bounds for the number of conjugacy classes of a finite group.
\newblock {\em J. Algebra}, 198(2):538--562, 1997.

\bibitem{MW}
D.~W. Masser and G.~W{\"u}stholz.
\newblock Galois properties of division fields of elliptic curves.
\newblock {\em Bull. London Math. Soc.}, 25(3):247--254, 1993.

\bibitem{Mum}
David Mumford.
\newblock On the {K}odaira dimension of the {S}iegel modular variety.
\newblock In {\em Algebraic geometry---open problems ({R}avello, 1982)}, volume
  997 of {\em Lecture Notes in Math.}, pages 348--375. Springer, Berlin, 1983.

\bibitem{Schanuel}
Stephen~Hoel Schanuel.
\newblock Heights in number fields.
\newblock {\em Bull. Soc. Math. France}, 107(4):433--449, 1979.

\bibitem{Ser1}
Jean-Pierre Serre.
\newblock Propri\'et\'es galoisiennes des points d'ordre fini des courbes
  elliptiques.
\newblock {\em Invent. Math.}, 15(4):259--331, 1972.

\bibitem{Ser2}
Jean-Pierre Serre.
\newblock {\em Lectures on the {M}ordell-{W}eil theorem}.
\newblock Aspects of Mathematics, E15. Friedr. Vieweg \& Sohn, Braunschweig,
  1989.
\newblock Translated from the French and edited by Martin Brown from notes by
  Michel Waldschmidt.

\bibitem{Ser3}
Jean-Pierre Serre.
\newblock On a theorem of {J}ordan.
\newblock {\em Bull. Amer. Math. Soc. (N.S.)}, 40(4):429--440 (electronic),
  2003.

\bibitem{Silv}
Joseph~H. Silverman.
\newblock The theory of height functions.
\newblock In {\em Arithmetic geometry ({S}torrs, {C}onn., 1984)}, pages
  151--166. Springer, New York, 1986.

\bibitem{Vasiu}
Adrian Vasiu.
\newblock Surjectivity criteria for {$p$}-adic representations. {I}.
\newblock {\em Manuscripta Math.}, 112(3):325--355, 2003.

\bibitem{wang}
Wenxiang Wang.
\newblock On the moduli space of principally polarized abelian varieties.
\newblock In {\em Mapping class groups and moduli spaces of {R}iemann surfaces
  ({G}\"ottingen, 1991/{S}eattle, {WA}, 1991)}, volume 150 of {\em Contemp.
  Math.}, pages 361--365. Amer. Math. Soc., Providence, RI, 1993.

\bibitem{Zyw1}
David Zywina.
\newblock Elliptic curves with maximal {G}alois action on their torsion points.
\newblock {\em Bull. Lond. Math. Soc.}, 42(5):811--826, 2010.

\end{thebibliography}
\bibliographystyle{plain}
\end{document}